\documentclass[12pt]{amsart}
\usepackage{amsmath, amsfonts, amssymb, amsthm,hyperref,mathtools,array}
\usepackage[T1]{fontenc}
\hypersetup{hypertex=true,
	colorlinks=true,
	linkcolor=blue,
	anchorcolor=blue,
	citecolor=blue}
\usepackage{bm}
\allowdisplaybreaks[4]
\def\ord{{\rm ord}}

\def\u{{\bm u}}

\def\0{{\bm 0}}

\def\alg{{\rm alg}}

\def\bchi{{\bm \chi}}

\def\Ack{\medskip\noindent {\bf Acknowledgments}}

\theoremstyle{plain}
\newtheorem{theorem}{Theorem}[section]
\newtheorem{lemma}{Lemma}

\theoremstyle{definition}

\theoremstyle{remark}
\newtheorem{remark}{Remark}

\makeatletter
\@namedef{subjclassname@2020}{%
	\textup{2020} Mathematics Subject Classification}
\makeatother
\vspace{4mm}

\begin{document}
	
	\title[On Diophantine $m$-tuples related to primitive elements of finite fields]
	{On Diophantine $m$-tuples related to primitive elements of finite fields}
	\author[H.-L. Wu]{Hai-Liang Wu}
	
	\address {(Hai-Liang Wu) School of Science, Nanjing University of Posts and Telecommunications, Nanjing 210023, People's Republic of China}
	\email{\tt whl.math@smail.nju.edu.cn}

	\keywords{Diophantine tuples, character sums, finite fields.
		\newline \indent 2020 {\it Mathematics Subject Classification}. Primary 11T24, 11L40; Secondary 11T30.
		\newline \indent This research was supported by the Natural Science Foundation of China (Grant No. 12101321) and the Natural Science Foundation of the Higher Education Institutions of Jiangsu Province (Grant No. 25KJB110010).}
	
	\begin{abstract}
		Inspired by recent works on Diophantine tuples over finite fields, in this paper we consider Diophantine tuples related to primitive elements of finite fields. 	Let $\mathbb{F}_q$ be the finite field with $q$ elements and $\mathbb{F}_q^*=\mathbb{F}_q\setminus\{0\}$ be the multiplicative cyclic group of all non-zero elements over $\mathbb{F}_q$. An element $g\in\mathbb{F}_q$ is called primitive if $g$ generates the group $\mathbb{F}_q^*$. A set $\{x_1,x_2,\cdots,x_m\}\subseteq\mathbb{F}_q^*$ of $m$ elements is said to be a $\mathcal{P}$-Diophantine $m$-tuple over $\mathbb{F}_q$ if $x_ix_j+1$ is primitive for any $1\le i\le j\le m$. Let $N_m$ denote the number of $\mathcal{P}$-Diophantine tuples over $\mathbb{F}_q$. Then we obtain the asymptotic formula
		$$m!\cdot N_m=\left(\frac{\varphi(q-1)}{q-1}\right)^{m(m+1)/2}q^m+O_{m,r}\left(q^{m-\frac{1}{2}+r}\right),$$
		where $\varphi(\cdot)$ is the Euler totient function and $r\in(0, 1/2)$ is an arbitrary real number. Moreover, we prove that there exists a $\mathcal{P}$-Diophantine $m$-tuple over $\mathbb{F}_q$ whenever $q\ge \exp(\exp(m(m+1)))$. 
	\end{abstract}
	\maketitle
	
	\section{Introduction}
	\setcounter{lemma}{0}
	\setcounter{theorem}{0}
	\setcounter{equation}{0}
	\setcounter{conjecture}{0}
	\setcounter{remark}{0}
	\setcounter{corollary}{0}
	
	\subsection{Notation} In this paper, $q$ denotes a prime power and $\mathbb{F}_q$ is the finite field with $q$ elements. Let $\mathbb{F}_q^*=\mathbb{F}_q\setminus\{0\}$ be the cyclic group of all non-zero elements over $\mathbb{F}_q$. An element $g\in\mathbb{F}_q$ is said to be a primitive element of $\mathbb{F}_q$ if $g$ is a generator of $\mathbb{F}_q^*$. The set of all primitive elements of $\mathbb{F}_q$ is written as $\mathcal{P}_q$. 
	
	Also, we use $\widehat{\mathbb{F}_q^*}$ to denote the cyclic group of all multiplicative characters of $\mathbb{F}_q$, and let $\varepsilon$ denote the trivial character. Given any $\chi\in\widehat{\mathbb{F}_q^*}$, we additionally define $\chi(0)=0$ and let $\ord(\chi)$ be the order of $\chi$. 
	
	Finally, $\#S$ denotes the cardinality of a set $S$. 
	
	\subsection{Background and motivation} A set $\{x_1,x_2,\cdots,x_m\}\subseteq\mathbb{Z}^+$ of $m$ distinct positive integers is said to be a {\it Diophantine $m$-tuple} if $x_ix_j+1$ is a perfect square for any $1\le i<j\le m$. 
	
	The study of Diophantine $m$-tuples has a long and rich history. The first Diophantine $4$-tuple 
	$$\{1,3,8,120\}$$ 
	was found by Fermat. Later Euler extended this to a rational Diophantine $5$-tuple
	$$\left\{1,3,8,120, \frac{777480}{8288641}\right\}.$$
	A challenging conjecture in this topic is that there does not exist a Diophantine $5$-tuple. In 2004, Dujella made a breakthrough on this conjecture. Dujella \cite{Dujella04} showed that there are only finitely many Diophantine $5$-tuples and there are no Diophantine $6$-tuples. This long-standing conjecture was finally confirmed by He, Togb\'e and Ziegler \cite{HTZ} in 2019. Readers may refer to a recent book of Dujella \cite{Dujella24} for a comprehensive introduction to this topic.
	
	Inspired by the above results, we are naturally led to consider a generalization of the definition of Diophantine $m$-tuples to other algebraic structures. In this paper, we mainly focus on the Diophantine tuples over the finite field $\mathbb{F}_q$. A set $\{x_1,\cdots,x_m\}\subseteq\mathbb{F}_q^*$ of $m$ distinct non-zero elements is called a Diophantine $m$-tuple over $\mathbb{F}_q$ if $x_ix_j+1$ is a square over $\mathbb{F}_q$ for any $1\le i<j\le m$. Moreover, Dujella and Petri\v{c}evi\'c \cite{DP} posed the definition of {\it strong Diophantine $m$-tuples} over $\mathbb{F}_q$. More precisely, a set $\{x_1,\cdots,x_m\}\subseteq\mathbb{F}_q^*$ is said to be a strong Diophantine $m$-tuple if $x_ix_j+1$ is a square over $\mathbb{F}_q$ for any $1\le i\le j\le m$. Recently, research on Diophantine 
	$m$-tuples over finite fields has been very active. For example, Dujella and Kazalicki \cite{DK} showed that 
	$$m! \cdot N_m(\mathbb{F}_p)=2^{-m(m-1)/2}p^m+o(p^m),$$
	where $p$ is a prime and $N_m(\mathbb{F}_p)$ is the number of Diophantine $m$-tuples over $\mathbb{F}_p$. Later Mani and Rubinstein-Salzedo \cite{MR} extended this to an arbitrary finite field $\mathbb{F}_q$ and improved the error term, that is,
	$$m! \cdot N_m(\mathbb{F}_q)=2^{-m(m-1)/2}q^m+O(q^{m-1/2}).$$
	Recently, Shparlinski \cite{Shparlinski} further improved the error term and obtain 
	$$m! \cdot N_m(\mathbb{F}_q)=2^{-m(m-1)/2}q^m+O(q^{m-1}).$$
	From these results, it is clear that Diophantine $m$-tuples exist if $q$ is sufficiently large relative to $m$. By using an inductive method, Dujella and Kazalicki \cite[Theorem 17]{DK} proved that there is at least one Diophantine $m$-tuple over $\mathbb{F}_p$ whenever $p>2^{2m-2}m^2$. Recently, the generalizations of this type of problem have received considerable attention. Readers may refer to \cite{Gu,KYY,TY}. 
	
	On the other hand, the primitive elements are also important objects of study over finite fields, and their combinatorial properties have been extensively studied. For examples, Cohen and his collaborators \cite{Cohen15} showed that for any odd prime power $q>169$, there is an element $x\in\mathbb{F}_q$ such that $x,x+1,x+2$ are all primitive elements of $\mathbb{F}_q$. For more results concerning the combinatorial properties of primitive elements, readers may refer to \cite{Cohen03,Cohen10,Cohen15,Cohen21}.
	
	Motivated by the above results, in this paper, we introduce primitive elements into the study of Diophantine tuples. Analogous to the definition of strong Diophantine $m$-tuples posed by Dujella and Petri\v{c}evi\'c \cite{DP}, a set $\{x_1,x_2,\cdots,x_m\}\subseteq\mathbb{F}_q^*$ of $m$ distinct elements is called a {\it $\mathcal{P}$-Diophantine $m$-tuple} over $\mathbb{F}_q$ if $x_ix_j+1\in\mathcal{P}_q$ for any $1\le i\le j\le m$. 
	
	Let's introduce some examples here. With the help of a computer, one can verify  the following examples:
	
	\begin{itemize}
		\item for $p=11$, we have 
		$$\mathcal{P}_{11}=\{2,6,7,8\}\subseteq\mathbb{F}_{11}^*,$$
		and $\{1,7\}\subseteq \mathbb{F}_{11}^*$ is a $\mathcal{P}$-Diophantine $2$-tuple over $\mathbb{F}_{11}$;

		\item for $p=17$, we have $$\mathcal{P}_{17}=\{3,5,6,7,10,11,12,14\}\subseteq\mathbb{F}_{17}^*,$$
		and $\{2,3,15\}\subseteq\mathbb{F}_{17}^*$ is a $\mathcal{P}$-Diophantine $3$-tuple over $\mathbb{F}_{17}$;
		
		\item for $p=47$, we have 
		$$\mathcal{P}_{47}=\{5,10,11,13,15,19,20,22,23,26,29,30,31,33,35,38,39,40,41,43,44,45\}\subseteq\mathbb{F}_{47}^*,$$
		and $\{2,5,34,42\}\subseteq\mathbb{F}_{47}^*$ is a $\mathcal{P}$-Diophantine $4$-tuple over $\mathbb{F}_{47}$.
	\end{itemize}
	
	\subsection{The main theorem} Now we state our main result.
	
	\begin{theorem}\label{Thm. A}
		Let $m\ge 2$ be an integer and $N_m$ be the number of $\mathcal{P}$-Diophantine $m$-tuples over $\mathbb{F}_q$. Then, for any real number $r\in(0,1/2)$ we have the asymptotic formula
		$$m!\cdot N_m=\left(\frac{\varphi(q-1)}{q-1}\right)^{m(m+1)/2}q^m+O_{m,r}\left(q^{m-\frac{1}{2}+r}\right),$$
		where $\varphi(\cdot)$ is the Euler totient function. Moreover, let $M(m)=\exp(\exp(m(m+1)))$. Then, for any finite field $\mathbb{F}_q$ with $q\ge M(m)$, there exists a $\mathcal{P}$-Diophantine $m$-tuple over $\mathbb{F}_q$. 
	\end{theorem}
	
	\subsection{Outline of this paper} In Section 2, we will introduce some necessary lemmas. We shall prove our theorem in Section 3.
	
	\section{Preliminaries}
	\setcounter{lemma}{0}
	\setcounter{theorem}{0}
	\setcounter{equation}{0}
	\setcounter{conjecture}{0}
	\setcounter{remark}{0}
	\setcounter{corollary}{0}
	
	Let $\mathcal{P}_q$ be the set of all primitive elements of $\mathbb{F}_q$. We begin with the characteristic function $1_{\mathcal{P}_q}$ of $\mathcal{P}_q$ (cf. \cite{Cohen03,Cohen10,Cohen15,Cohen21}). 
	
	\begin{lemma}\label{Lem. characteristic function of P}
		Let $\mathcal{P}_q$ be the set of all primitive elements of the finite field $\mathbb{F}_q$. Then 
		$$1_{\mathcal{P}_q}(x)=\theta_{q-1}\sum_{d\mid q-1}\frac{\mu(d)}{\varphi(d)}\sum_{\substack{\chi\in\widehat{\mathbb{F}_q^*} \\ \ord(\chi)=d}} \chi(x)=\begin{cases}
		1  & \mbox{if}\ x\in\mathcal{P}_q,\\
		0 & \mbox{otherwise},
		\end{cases}$$
		where $\mu(\cdot)$ is the M\"obius function, $\varphi(\cdot)$ is the Euler totient function, and $\theta_{q-1}=\varphi(q-1)/(q-1)$. 
	\end{lemma}
	
	Next we introduce the Weil bound (cf. \cite[Theorem 5.41]{LN}).
	
	\begin{lemma}\label{Lem. the Weil Bound}
		Let $\psi\in\widehat{\mathbb{F}_q^*}$ with $\ord(\psi)=d>1$, and let $f(t)\in\mathbb{F}_q[T]$ be a monic polynomial with $f(T)\neq g(T)^d$ for any $g(t)\in\mathbb{F}_q[T]$. Then, for any $a\in\mathbb{F}_q$ we have 
		$$\left|\sum_{x\in\mathbb{F}_q}\psi(af(x))\right|\le (r-1)\sqrt{q},$$
		where $r$ is the number of distinct roots of $f(T)$ in an algebraic closure $\mathbb{F}_q^{\alg}$ of $\mathbb{F}_q$. 
	\end{lemma}
	
	For any positive integer $n$, let 
	$$\omega_n=\#\left\{p: p\mid n\ \text{and}\ p\ \text{is a prime}\right\}$$
	be the number of all distinct prime divisors of $n$. We need the following result due to Robin \cite[Theorem 11]{Robin}.
	
	\begin{lemma}\label{Lem. bound due to Robin}
		For any integer $n\ge 3$, we have 
		$$\omega_n<1.3841\frac{\log n}{\log\log n}.$$
	\end{lemma}
	
	\begin{remark}\label{Rem. bound for W(q-1)}
		Let 
		$$W_{q-1}=\#\left\{d\mid q-1: d\ge 1\ \text{and}\ d\ \text{is square-free}\right\}$$
		be the number of positive square-free divisors of $q-1$. Then clearly $\log W_{q-1}=\omega_{q-1}\cdot \log 2$. Applying Lemma \ref{Lem. bound due to Robin}, for any prime power $q\ge 4$ we obtain 
		$$W_{q-1}=2^{\omega_{q-1}}=e^{\omega_{q-1}\cdot \log 2}\le (q-1)^{\frac{1.3841\cdot\log 2}{\log\log(q-1)}}.$$
		Thus, for any positive real number $r$, as $q\rightarrow\infty$, we have 
		\begin{equation}\label{Eq. bound for W(q-1)}
			W_{q-1}\ll_r q^r.
		\end{equation}
	\end{remark}
	
	We conclude this section with a classical result due to Rosser and Schoenfeld \cite[Theorem 15]{RS}.
	
	\begin{lemma}\label{Lem. on low bound for theta}
		For any positive integer $n\ge 3$, we have 
		$$\frac{n}{\varphi(n)}<e^C\cdot \log\log n+\frac{2.50637}{\log\log n},$$
		where $C$ is the Euler constant. 
	\end{lemma}

    \section{Proof of Theorem \ref{Thm. A}}
	\setcounter{lemma}{0}
	\setcounter{theorem}{0}
	\setcounter{equation}{0}
	\setcounter{conjecture}{0}
	\setcounter{remark}{0}
	\setcounter{corollary}{0}
	
	We begin with the following lemma.
	
	\begin{lemma}\label{Lem. an inequality in proof of Thm. A}
		Let $m\ge 2$ be an integer and $q$ a prime power. Suppose $q\ge \exp(\exp(m(m+1)))$. Then 
		$$\sqrt{q}\ge (m+1)\cdot W_{q-1}^{m(m+1)/2}.$$
	\end{lemma}
	
	\begin{proof}
		First we claim that 
		\begin{equation}\label{Eq. a in the proof of Lem. an inequality in proof of Thm. A}
			\frac{2\log(m+1)}{e^{m(m+1)}}\le 0.03
		\end{equation}
		for any integer $m\ge 2$. In fact, by a computer t is easy to verify that (\ref{Eq. a in the proof of Lem. an inequality in proof of Thm. A}) holds for $m=2$. Suppose now $m\ge 3$. As $\log(m+1)\le m$ and $e^{m(m+1)}=(e^{m+1})^m\ge 2^{4m}=16^m$, we obtain
		$$\frac{2\log(m+1)}{e^{m(m+1)}}\le a_m=\frac{2m}{16^m}$$
		for any $m\ge 3$. Noting that 
		$$a_{m+1}/a_m=\frac{1}{16}\left(1+\frac{1}{m}\right)<1,$$
		by the above we have 
		$$\frac{2\log(m+1)}{e^{m(m+1)}}\le a_m\le a_3=6/16^3\le 0.03.$$
		Thus, the inequality (\ref{Eq. a in the proof of Lem. an inequality in proof of Thm. A}) holds. On the other hand, let $f(x)=\log x/(\log\log x)$. Then it is easy to see that $f(x)$ is strictly increasing for all $x>e^e$ (note that $15\le e^e\le 16$). Thus, by Lemma \ref{Lem. bound due to Robin} we obtain 
		\begin{equation}\label{Eq. b in the proof of Lem. an inequality in proof of Thm. A}
			\omega_{q-1}<1.3841\frac{\log (q-1)}{\log\log (q-1)}<1.3841\frac{\log q}{\log\log q}
		\end{equation}
		whenever $q\ge 17$. 
		
		Suppose $q\ge \exp(\exp(m(m+1)))\ge 17$. Then, by (\ref{Eq. a in the proof of Lem. an inequality in proof of Thm. A}) and (\ref{Eq. b in the proof of Lem. an inequality in proof of Thm. A}) one can verify that 
		\begin{align*}
			1
&\ge 0.03+0.96\\
&\ge \frac{2\log(m+1)}{e^{m(m+1)}}+1.9188\cdot \frac{m(m+1)/2}{m(m+1)}\\
&\ge \frac{2\log(m+1)}{\log q}+1.9188\cdot \frac{m(m+1)/2}{\log\log q}\\
&\ge \frac{2\log(m+1)}{\log q}+\log 2\cdot 1.3841\cdot 2 \cdot \frac{\log q\cdot m(m+1)/2}{\log q\cdot \log\log q}\\
&\ge \frac{2\log(m+1)}{\log q}+\log 2\cdot \omega_{q-1}\cdot 2 \cdot \frac{m(m+1)/2}{\log q}\\
&=\frac{2\log(m+1)}{\log q}+2\log W_{q-1} \cdot \frac{m(m+1)/2}{\log q}.
		\end{align*}
	This implies 
	$$\frac{\log q}{2}-\log(m+1)-\frac{m(m+1)}{2}\log W_{q-1}\ge 0,$$
	i.e., 
	$$\sqrt{q}\ge (m+1)\cdot W_{q-1}^{m(m+1)/2}.$$
	This completes the proof.		
	\end{proof}
	
	The following notations will be used frequently later. For any positive integer $r$, let 
	$$\mathcal{D}_r=\left\{(x_1,x_2,\cdots,x_r): x_1,x_2,\cdots,x_r\in\mathbb{F}_q^*\ \text{are mutually distinct}\right\}.$$
	Also, we use the symbol ${\bchi}=(\chi_{ij})_{1\le i\le j\le m}$ (or simply $\bchi=(\chi_{ij})$) to denote the formal upper triangular matrix
	$$\begin{pmatrix}
		\chi_{11}   &  \chi_{12}    &  \cdots  &  \chi_{1m}\\
		0                 &  \chi_{22}   &  \cdots  &  \chi_{2m}\\
		\vdots       & \vdots          &  \ddots  &  \vdots\\
		0                &      0              &  \cdots   &  \chi_{mm}\\
 	\end{pmatrix},$$
 	where $\chi_{ij}\in\widehat{\mathbb{F}_q^*}$ for any $1\le i\le j\le m$. In particular, let 
 	$$\bchi_0=\begin{pmatrix}
 	\varepsilon  &  \varepsilon   &  \cdots  &  \varepsilon\\
 		0                &  \varepsilon  &  \cdots   &  \varepsilon\\
 		\vdots       & \vdots            &  \ddots  &  \vdots\\
 		0                &      0                 &  \cdots   &  \varepsilon\\
 	\end{pmatrix},$$
  The set of all these formal upper triangular matrices is written as $\mathcal{M}$. 
 	
 	Now we are in a position to prove our theorem.
	
	{\bfseries\noindent Proof of Theorem \ref{Thm. A}}. We will divide the proof into six parts.
	
	{\bfseries\noindent Part 1}: Setup of the counting function.
	
	By Lemma \ref{Lem. characteristic function of P} one can verify that 
	$$1_{\mathcal{P}_q}(x)=\theta_{q-1}\sum_{d\mid q-1}\frac{\mu(d)}{\varphi(d)}\sum_{\substack{\chi\in\widehat{\mathbb{F}_q^*} \\ \ord(\chi)=d}} \chi(x)=\theta_{q-1}\sum_{\chi\in\widehat{\mathbb{F}_q^*}}c_{\chi}\cdot\chi(x),$$
	where $c_{\chi}=\mu(\ord(\chi))/\varphi(\ord(\chi))$. Thus, if $N_m$ denotes the number of $\mathcal{P}$-Diophantine $m$-tuples over $\mathbb{F}_q$, then 
	\begin{align}\label{Eq. setup of the counting function}
		m!\cdot N_m
&=\sum_{(x_1, \cdots, x_m)\in\mathcal{D}_m}\prod_{1\le i\le j\le m}1_{\mathcal{P}_q}(x_ix_j+1)\notag\\
&=\theta_{q-1}^{m(m+1)/2}\sum_{(x_1, \cdots, x_m)\in\mathcal{D}_m}\prod_{1\le i\le j\le m}\sum_{\chi\in\widehat{\mathbb{F}_q^*}}\chi(x_ix_j+1)\notag\\
&=\theta_{q-1}^{m(m+1)/2}\sum_{(x_1, \cdots, x_m)\in\mathcal{D}_m}\sum_{\bchi=(\chi_{ij})\in\mathcal{M}}\prod_{1\le i\le j\le m}c_{\chi_{ij}}\cdot\chi_{ij}(x_ix_j+1)\notag\\
&=\theta_{q-1}^{m(m+1)/2}\sum_{\bchi=(\chi_{ij})\in\mathcal{M}}\prod_{1\le i\le j\le m}c_{\chi_{ij}}\sum_{(x_1, \cdots, x_m)\in\mathcal{D}_m}\prod_{1\le i\le j\le m}\chi_{ij}(x_ix_j+1)\notag\\
&=\theta_{q-1}^{m(m+1)/2}\sum_{\bchi=(\chi_{ij})\in\mathcal{M}}c_{\bchi}\cdot S(\bchi),
	\end{align}
where 
$$c_{\bchi}=\prod_{1\le i\le j\le m}c_{\chi_{ij}},$$
and 
$$S(\bchi)=\sum_{(x_1, \cdots, x_m)\in\mathcal{D}_m}\prod_{1\le i\le j\le m}\chi_{ij}(x_ix_j+1).$$
	
    {\bfseries\noindent Part 2}: An estimate for $S(\bchi_0)$.
	
	Let $(\mathbb{F}_q)^m=\{(x_1,\cdots,x_m): x_1,\cdots,x_m\in\mathbb{F}_q\}$ and define 
	\begin{align*}
		E_1&=\left\{(x_1,\cdots,x_m)\in(\mathbb{F}_q)^m: x_i=0\ \text{for some}\ 1\le i\le m\right\},\\
		E_2&=\left\{(x_1,\cdots,x_m)\in(\mathbb{F}_q)^m: x_i=x_j\ \text{for some}\ 1\le i<j\le m\right\},\\
		E_3&=\left\{(x_1,\cdots,x_m)\in(\mathbb{F}_q)^m: x_ix_j+1=0\ \text{for some}\ 1\le i<j\le m\right\},\\
		E_4&=\left\{(x_1,\cdots,x_m)\in(\mathbb{F}_q)^m: x_i^2+1=0\ \text{for some}\ 1\le i\le m\right\}.
	\end{align*}
	Then clearly $\#E_1\le mq^{m-1}$, $\#E_2\le m(m-1)q^{m-1}/2$, $\#E_3\le m(m-1)q^{m-1}/2$, and $\#E_4\le 2mq^{m-1}$. Applying this to $S(\bchi_0)$, we obtain 
	\begin{align}\label{Eq. estimate for the main term}
		S(\bchi_0)
&=\sum_{(x_1, \cdots, x_m)\in\mathcal{D}_m}\prod_{1\le i\le j\le m}\varepsilon(x_ix_j+1)\notag\\
&=\#(\mathbb{F}_q)^m\setminus\bigcup_{1\le j\le 4}E_j\notag\\
&\ge \#(\mathbb{F}_q)^m-\sum_{\le j\le 4}\#E_j\notag\\
&\ge q^m-q^{m-1}(m^2+2m).
	\end{align}
	
{\bfseries\noindent Part 3}: An estimate for $S(\bchi)$ with $\bchi\neq \bchi_0$. 
	
Given an element $\bchi=(\chi_{ij})_{1\le i,j\le m}\in\mathcal{M}\setminus\{\bchi_0\}$, let 
$$t=\max\left\{1\le k\le m: \chi_{ik}\neq\varepsilon\ \text{for some $1\le i\le k$ or}\  \chi_{kj}\neq\varepsilon\ \text{for some $k\le j\le m$} \right\}.$$

We first consider the sum 
\begin{equation}\label{Eq. definition of S star}
	S^*(\bchi)=\sum_{(x_1, \cdots, x_t)\in\mathcal{D}_t}\prod_{1\le i\le j\le t}\chi_{ij}(x_ix_j+1)\sum_{x_{t+1}, \cdots, x_m\in\mathbb{F}_q}1.
\end{equation}
Suppose $t=1$. Then 
$$S^*(\bchi)=q^{m-1}\sum_{x_1\in\mathbb{F}_q^*}\chi_{11}(x_1^2+1),$$
where $\chi_{11}\in\widehat{\mathbb{F}_q^*}$ is a non-trivial character. If $2\mid q$, then $T^2+1=(T+1)^2$ and $\ord(\chi_{11})$ is odd since $\ord(\chi_{11})\mid q-1$. If $2\nmid q$, then the polynomial $T^2+1$ has two distinct roots in $\mathbb{F}_q^{\alg}$. Thus, $T^2+1\neq g(T)^{\ord(\chi_{11})}$ for any $g(t)\in\mathbb{F}_q[T]$. Applying Lemma \ref{Lem. the Weil Bound}, we obtain 
$$|S^*(\bchi)|\le q^{m-1}\left|-1+\sum_{x_1\in\mathbb{F}_q}\chi_{11}(x_1^2+1)\right|\le q^{m-1}+q^{m-1/2}.$$

Now suppose $t\ge 2$. We call a vector $\u=(x_1, x_2, \cdots, x_{t-1})\in\mathcal{D}_{t-1}$ good if $x_i^2+1\neq 0$ for any $1\le i\le t-1$; otherwise $\u$ is said to be bad. Let $\mathcal{D}_{t-1}^{\rm good}$ be the set of all good vectors and $\mathcal{D}_{t-1}^{\rm bad}$ be the set of all bad vectors. Then it is clear that 
\begin{equation}\label{Eq. number of good tuples}
	\#\mathcal{D}_{t-1}^{\rm good}=\#\left\{\u=(x_1,\cdots,x_{t-1})\in\mathcal{D}_{t-1}: \u\ \text{is good}\right\}\le q^{t-1},
\end{equation}
and 
\begin{equation}\label{Eq. number of bad tuples}
	\#\mathcal{D}_{t-1}^{\rm bad}=\#\left\{\u=(x_1,\cdots,x_{t-1})\in\mathcal{D}_{t-1}: \u\ \text{is bad}\right\}\le 2(t-1)q^{t-2}.
\end{equation}
For a good $\u=(x_1,\cdots, x_{t-1})\in\mathcal{D}_{t-1}$, let $d$ be the least common multiple of 
$$\ord(\chi_{1t}),\ord(\chi_{2t}),\cdots,\ord(\chi_{tt}).$$
Since at least one of $\chi_{1t},\cdots,\chi_{tt}$ is non-trivial, we have $d>1$.  Hence, fixing a character $\psi$ with $\ord(\psi)=d$, for any $1\le i\le t$ we may set $\chi_{it}=\psi^{b_i}$ for some integer $0\le b_i\le d-1$. Since $b_1,b_2,\cdots,b_t$ are not all zero and $x_i^2+1\neq 0$ for any $1\le i\le t-1$, we see that 
$$-x_1^{-1}, \cdots, -x_{t-1}^{-1}\in\mathbb{F}_q^{\alg}\setminus\{\pm\sqrt{-1}\}$$ 
are mutually distinct, where $\sqrt{-1}\in\mathbb{F}_q^{\alg}$ with $(\sqrt{-1})^2=-1\in\mathbb{F}_q$. When $b_t>0$, as in the case $t=1$, it is easy to verify that $(T^2+1)^{b_t}\neq h(T)^d$ for any $h(T)\in\mathbb{F}_q[T]$. Applying the above discussions, we see that 
$$f_{\u}(T)=(T^2+1)^{b_t}\prod_{1\le i\le t-1}(x_iT+1)^{b_i}\neq g(T)^d$$
for any $g(T)\in\mathbb{F}_q[T]$. Thus, applying Lemma \ref{Lem. the Weil Bound}, for good vector $\u=(x_1,\cdots,x_{t-1})$ we have 
\begin{equation}\label{Eq. bound for sum related to good tuples}
	\left|\sum_{x_t\in\mathbb{F}_q}\prod_{i=1}^t\chi_{it}(x_ix_t+1)\right|
=\left|\sum_{x_t\in\mathbb{F}_q}\prod_{i=1}^t\psi^{b_i}(x_ix_t+1)\right|\\
=\left|\sum_{x_t\in\mathbb{F}_q}\psi(f_{\u}(x_t))\right|\\
\le tq^{1/2}.
\end{equation}
From (\ref{Eq. definition of S star}) one can verify that 
\begin{align*}
	\left|S^*(\bchi)\right|
&=q^{m-t}\left|\sum_{(x_1, \cdots, x_t)\in\mathcal{D}_t}\prod_{1\le i\le j\le t}\chi_{ij}(x_ix_j+1)\right|\\
&=q^{m-t}\left|\sum_{(x_1, \cdots, x_{t-1})\in\mathcal{D}_{t-1}}\prod_{1\le i\le j\le t-1}\chi_{ij}(x_ix_j+1)\sum_{x_t\not\in\{0,x_1,\cdots,x_{t-1}\}}\prod_{1\le i\le t}\chi_{it}(x_ix_t+1)\right|\\
&\le q^{m-t}\sum_{(x_1, \cdots, x_{t-1})\in\mathcal{D}_{t-1}}\left|\sum_{x_t\not\in\{0,x_1,\cdots,x_{t-1}\}}\prod_{1\le i\le t}\chi_{it}(x_ix_t+1)\right|\\
&=S_{\rm good}^*+S_{\rm bad}^*,
\end{align*}
where 
$$S_{\rm good}^*=q^{m-t}\sum_{(x_1, \cdots, x_{t-1})\in\mathcal{D}_{t-1}^{\rm good}}\left|\sum_{x_t\not\in\{0,x_1,\cdots,x_{t-1}\}}\prod_{1\le i\le t}\chi_{it}(x_ix_t+1)\right|,$$
and 
$$S_{\rm bad}^*=q^{m-t}\sum_{(x_1, \cdots, x_{t-1})\in\mathcal{D}_{t-1}^{\rm bad}}\left|\sum_{x_t\not\in\{0,x_1,\cdots,x_{t-1}\}}\prod_{1\le i\le t}\chi_{it}(x_ix_t+1)\right|.$$
For $S_{\rm good}^*$, assembling (\ref{Eq. number of good tuples}) and (\ref{Eq. bound for sum related to good tuples}) gives 
\begin{align}\label{Eq. bound for S star good}
	S_{\rm good}^*
&\le q^{m-t}\sum_{(x_1, \cdots, x_{t-1})\in\mathcal{D}_{t-1}^{\rm good}}\left(t+\left|\sum_{x_t\in\mathbb{F}_q}\prod_{1\le i\le t}\chi_{it}(x_ix_t+1)\right|\right)\notag\\
&\le q^{m-t}\cdot \#\mathcal{D}_{t-1}^{\rm good}\cdot \left(tq^{1/2}+t\right)\notag\\
&\le q^{m-1}\left(t+tq^{1/2}\right)\notag\\
&\le q^{m-1}(m+mq^{1/2}).
\end{align}
For $S_{\rm bad}^*$, applying (\ref{Eq. number of bad tuples}) we obtain 
\begin{align}\label{Eq. bound for S star bad}
	S_{\rm bad}^*
&\le \sum_{(x_1, \cdots, x_{t-1})\in\mathcal{D}_{t-1}^{\rm bad}}\sum_{x_t\not\in\{0,x_1,\cdots,x_{t-1}\}}\left|\prod_{1\le i\le t}\chi_{it}(x_ix_t+1)\right|\notag\\
&\le q^{m-t}\cdot \#\mathcal{D}_{t-1}^{\rm bad}\cdot (q-t)\notag\\
&\le q^{m-t}\cdot 2(t-1)q^{t-2}\cdot (q-t)\notag\\
&\le 2mq^{m-1}.
\end{align}
By (\ref{Eq. bound for S star good}), (\ref{Eq. bound for S star bad}) and the case $t=1$, we obtain 
\begin{equation}\label{Eq. bound for S star}
	\left|S^*(\bchi)\right|\le mq^{m-1/2}+3mq^{m-1}
\end{equation}
for any $1\le t\le m$. 
	
	Next we consider $|S^*(\bchi)-S(\bchi)|$. For any $\u=(x_1,\cdots,x_t)\in\mathcal{D}_t$, let 
	\begin{align*}
		E_1(\u)&=\left\{(x_{t+1},\cdots,x_m)\in(\mathbb{F}_q)^{m-t}: x_i=0\ \text{for some}\ t+1\le i\le m\right\},\\
		E_2(\u)&=\left\{(x_{t+1},\cdots,x_m)\in(\mathbb{F}_q)^{m-t}: x_i=x_j\ \text{for some}\ 1\le i<j\le m\ \text{with}\ j>t\right\},\\
		E_3(\u)&=\left\{(x_{t+1},\cdots,x_m)\in(\mathbb{F}_q)^{m-t}: x_ix_j+1=0\ \text{for some}\ 1\le i<j\le m\ \text{with}\ j>t\right\},\\
		E_4(\u)&=\left\{(x_{t+1},\cdots,x_m)\in(\mathbb{F}_q)^{m-t}: x_j^2+1=0\ \text{for some}\ t+1\le j\le m\right\}.
	\end{align*}
	Then it is clear that 
	$\#E_1(\u)\le (m-t)q^{m-t-1}\le mq^{m-t-1}$, $\#E_2(\u)\le m(m-1)q^{m-t-1}/2$, $\#E_3(\u)\le m(m-1)q^{m-t-1}/2$, and $\#E_4(\u)\le 2(m-t)q^{m-t-1}\le 2mq^{m-t-1}$. Let
	$$I_{\u}(x_{t+1},\cdots,x_m)=\begin{cases}
		1  & \mbox{if}\ (x_{t+1},\cdots,x_m)\in E_1(\u)\cup E_2(\u) \cup E_3(\u) \cup E_4(\u),\\
		0 & \mbox{otherwise}.
	\end{cases}$$
	Then, applying the above discussions, by (\ref{Eq. definition of S star}) one can verify that 
	\begin{align}\label{Eq. bound for S star -S}
		\left|S^*(\bchi)-S(\bchi)\right|
&=\left|\sum_{(x_1, \cdots, x_t)\in\mathcal{D}_t}\prod_{1\le i\le j\le t}\chi_{ij}(x_ix_j+1)\sum_{(x_{t+1},\cdots,x_m)\in(\mathbb{F}_q)^{m-t}}I_{\u}(x_{t+1},\cdots,x_m)\right|\notag\\
&\le \sum_{(x_1, \cdots, x_t)\in\mathcal{D}_t}\left|\prod_{1\le i\le j\le t}\chi_{ij}(x_ix_j+1)\sum_{(x_{t+1},\cdots,x_m)\in(\mathbb{F}_q)^{m-t}}I_{\u}(x_{t+1},\cdots,x_m)\right|\notag\\
&\le  \sum_{(x_1, \cdots, x_t)\in\mathcal{D}_t}\#\bigcup_{1\le j\le 4}E_j(\u)\notag\\
&\le \sum_{(x_1, \cdots, x_t)\in\mathcal{D}_t}\sum_{1\le j\le 4}\#E_j(\u)\notag\\
&=\#\mathcal{D}_t\sum_{1\le j\le 4}\#E_j(\u)\notag\\
&\le q^t(3mq^{m-t-1}+m(m-1)q^{m-t-1})\notag\\
&=(m^2+2m)q^{m-1}.
	\end{align}
	
Now assembling (\ref{Eq. bound for S star}) and (\ref{Eq. bound for S star -S}) gives 
\begin{equation}\label{Eq. bound for S with nontrivial chi}
	\left|S(\bchi)\right|\le \left|S^*(\bchi)-S(\bchi)\right|+\left|S^*(\bchi)\right|\le mq^{m-1/2}+(m^2+5m)q^{m-1}
\end{equation}
for any $\bchi\neq\bchi_0$. 

{\bfseries\noindent Part 4}: an estimate for $N_m$. 
	
	Recall that $c_{\bchi}$ is defined in (\ref{Eq. setup of the counting function}). We first consider the sum 
	$$\sum_{\bchi\in\mathcal{M}}|c_{\bchi}|.$$ 
	For any positive divisor $d$ of $q-1$, it is easy to see that 
	$$\#\left\{\chi\in\widehat{\mathbb{F}_q^*}: \ord(\chi)=d\right\}=\varphi(d).$$
	Applying this, one can verify that 
	\begin{align}\label{Eq. bound for sum of c chi}
		\sum_{\bchi=(\chi_{ij})\in\mathcal{M}}|c_{\bchi}|
&=\sum_{\bchi=(\chi_{ij})\in\mathcal{M}}\prod_{1\le i\le j\le m}|c_{\chi_{ij}}|\notag\\
&=\prod_{1\le i\le j\le m}\sum_{\chi_{ij}\in\widehat{\mathbb{F}_q^*}}|c_{\chi_{ij}}|\notag\\
&=\prod_{1\le i\le j\le m}\sum_{d\mid q-1}\frac{|\mu(d)|}{\varphi(d)}\#\left\{\chi\in\widehat{\mathbb{F}_q^*}: \ord(\chi)=d\right\}\notag\\
&=\prod_{1\le i\le j\le m}\sum_{d\mid q-1}|\mu(d)|\notag\\
&=\prod_{1\le i\le j\le m}W_{q-1}\notag\\
&=W_{q-1}^{m(m+1)/2}.
	\end{align}
	
	Now applying (\ref{Eq. estimate for the main term}), (\ref{Eq. bound for S with nontrivial chi}) and (\ref{Eq. bound for sum of c chi}) to (\ref{Eq. setup of the counting function}), we obtain 
	\begin{align}\label{Eq. estimate for Nm}
		m!\cdot N_m/\theta_{q-1}^{m(m+1)/2}
&=\sum_{\bchi\in\mathcal{M}}c_{\bchi}\cdot S(\bchi)\notag\\
&=S(\bchi_0)+\sum_{\bchi\in\mathcal{M}\setminus\{\bchi_0\}}c_{\bchi}\cdot S(\bchi)\notag\\
&\ge |S(\bchi_0)|-\sum_{\bchi\in\mathcal{M}\setminus\{\bchi_0\}}\left|c_{\bchi}\cdot S(\bchi)\right|\notag\\
&\ge q^m-q^{m-1}(m^2+2m)-\left(mq^{m-1/2}+(m^2+5m)q^{m-1}\right)\sum_{\bchi\in\mathcal{M}}|c_{\bchi}|\notag\\
&= q^m-q^{m-1}(m^2+2m)-\left(mq^{m-1/2}+(m^2+5m)q^{m-1}\right)W_{q-1}^{m(m+1)/2}\notag\\
&=q^{m-1}\left(q-(m^2+2m)-(m\sqrt{q}+m^2+5m)W_{q-1}^{m(m+1)/2}\right).
	\end{align}
	
{\bfseries\noindent Part 5}: the asymptotic formula for $N_m$.
	
From the above results, we see that 
$$m!\cdot N_m=\theta_{q-1}^{m(m+1)/2}\cdot S(\bchi_0)+\theta_{q-1}^{m(m+1)/2}\cdot\sum_{\bchi\in\mathcal{M}\setminus\{\bchi_0\}}c_{\bchi}\cdot S(\bchi).$$

For the term $\theta_{q-1}^{m(m+1)/2}\cdot S(\bchi_0)$, it follows from (\ref{Eq. estimate for the main term}) that
\begin{equation}\label{Eq. main term in asymptotic formula}
	\theta_{q-1}^{m(m+1)/2}\cdot S(\bchi_0)=\theta_{q-1}^{m(m+1)/2}\cdot q^m+O_m(q^{m-1}).
\end{equation}

Next we turn to the term $\theta_{q-1}^{m(m+1)/2}\cdot\sum_{\bchi\in\mathcal{M}\setminus\{\bchi_0\}}c_{\bchi}\cdot S(\bchi)$. Note that $\theta_{q-1}\le 1$. By (\ref{Eq. bound for S with nontrivial chi}), (\ref{Eq. bound for sum of c chi}) and (\ref{Eq. bound for W(q-1)}), for any real numbers $r\in (0, 1/2)$, when $q$ is sufficiently large, one can verify that 
\begin{align}\label{Eq. error term in asymptotic formula}
	\left|\theta_{q-1}^{m(m+1)/2}\cdot\sum_{\bchi\in\mathcal{M}\setminus\{\bchi_0\}}c_{\bchi}\cdot S(\bchi)\right|
&\le \left(mq^{m-1/2}+(m^2+5m)q^{m-1}\right)W_{q-1}^{m(m+1)/2}\notag\\
&\ll_{m,r}q^{m-\frac{1}{2}+r}.
\end{align}

Now we consider the quotient $q^{m-\frac{1}{2}+r}/(\theta_{q-1}^{m(m+1)/2}\cdot q^m)$, where $r\in(0, 1/2)$. Applying Lemma \ref{Lem. on low bound for theta}, when $q$ is large enough, we have the inequality 
$$\frac{1}{\theta_{q-1}}=\frac{q-1}{\varphi(q-1)}<e^C\cdot \log\log (q-1)+\frac{2.50637}{\log\log (q-1)}<2e^C\cdot \log\log (q-1)<2e^C\cdot \log\log q.$$
From this, one can verify that 
\begin{align*}
	\frac{q^{m-\frac{1}{2}+r}}{\theta_{q-1}^{m(m+1)/2}\cdot q^m}<(2e^{C})^{m(m+1)/2}\cdot \frac{(\log\log q)^{m(m+1)/2}}{q^{\frac{1}{2}-r}}.
\end{align*}
Since $0<r<1/2$, the above inequality implies that 
$$\lim_{q\rightarrow\infty}\frac{q^{m-\frac{1}{2}+r}}{\theta_{q-1}^{m(m+1)/2}q^m}=0.$$
Combining this with (\ref{Eq. main term in asymptotic formula}) and (\ref{Eq. error term in asymptotic formula}), we obtain the asymptotic formula
$$m!\cdot N_m=\left(\frac{\varphi(q-1)}{q-1}\right)^{m(m+1)/2}q^m+O_{m,r}\left(q^{m-\frac{1}{2}+r}\right).$$

{\bfseries\noindent Part 6}: the existence of $\mathcal{P}$-Diophantine $m$-tuples over $\mathbb{F}_q$.

First we prove that 
\begin{equation}\label{Eq. an elementary inequality involving m}
	(m+1)2^{m(m+1)/2}>2m^2+7m
\end{equation}
for any integer $m\ge 2$. In fact, (\ref{Eq. an elementary inequality involving m}) holds trivially for $m=2$. Suppose now $m\ge 3$. By induction on $m$, it is easy to verify that $4^m\ge m^3$ for any integer $m\ge 3$. Applying this we obtain 
$$(m+1)\cdot 2^{m(m+1)/2}\ge 4\cdot 2^{2m}\ge 4m^3=m(2m^2)+(2m^2)m>2m^2+7m.$$
Hence (\ref{Eq. an elementary inequality involving m}) holds for any $m\ge 2$.

Suppose now $q\ge \exp(\exp(m(m+1)))$. Then by Lemma \ref{Lem. an inequality in proof of Thm. A} we have 
$$\sqrt{q}\ge (m+1)W_{q-1}^{m(m+1)/2}.$$
Combining this with (\ref{Eq. an elementary inequality involving m}) and noting that $W_{q-1}\ge2$, one can verify that 
	\begin{align*}
	    &q-(m^2+2m)-(m\sqrt{q}+m^2+5m)W_{q-1}^{m(m+1)/2}\\
\ge &(m+1)\sqrt{q}W_{q-1}^{m(m+1)/2}-(m^2+2m)-(m\sqrt{q}+m^2+5m)W_{q-1}^{m(m+1)/2}\\
  =  &W_{q-1}^{m(m+1)/2}\left(\sqrt{q}-\left((m^2+5m)+\frac{m^2+2m}{W_{q-1}^{m(m+1)/2}}\right)\right)\\
 \ge&W_{q-1}^{m(m+1)/2}\left((m+1)W_{q-1}^{m(m+1)/2}-(2m^2+7m)\right)\\
 \ge &W_{q-1}^{m(m+1)/2}\left((m+1)2^{m(m+1)/2}-(2m^2+7m)\right)\\
    > &0.
	\end{align*}
	Applying this to (\ref{Eq. estimate for Nm}), we obtain $N_m>0$ whenever $q\ge \exp(\exp(m(m+1)))$.
	
	In view of the above, we have completed the proof of Theorem \ref{Thm. A}. \qed

	\Ack\ This research was supported by the Natural Science Foundation of China (Grant No. 12101321) and the Natural Science Foundation of the Higher Education Institutions of Jiangsu Province (Grant No. 25KJB110010).


\begin{thebibliography}{99}
	
	\bibitem{Cohen03} S. D. Cohen and S. Huczynska, The primitive normal basis theorem --Without a	computer, J. Lond. Math. Soc. 67 (2003), 41--56.
	
	\bibitem{Cohen10} S. D. Cohen and S. Huczynska, The strong primitive normal basis theorem, Acta Arith. 143 (2010), 299--332.
		
	\bibitem{Cohen15} S. D. Cohen, T. Oliveira e Silva and T. Trudgian, On consecutive primitive elements in a finite field, Bull. London Math. Soc. 47 (2015), 418--426.
	
	\bibitem{Cohen21} S. D. Cohen, H. Sharma and R. Sharma, Primitive values of rational functions at primitive elements of a finite field, J. Number Theory 219 (2021), 237--246.
	
	\bibitem{Dujella04} A. Dujella, There are only finitely many Diophantine quintuples, J. Reine Angew. Math. 566 (2004), 183--214.
	
	\bibitem{Dujella24} A. Dujella, Diophantine $m$-tuples and elliptic curves, Developments in Mathematics, vol. 79, Springer, 2024.
	
	\bibitem{DK} A. Dujella and M. Kazalicki, Diophantine $m$-tuples in finite fields and modular forms, Res. Number Theory 7 (2021), 1--24.
	
	\bibitem{DP} A. Dujella and V. Petri\v{c}evi\'c, Strong Diophantine triples, Exp. Math. 17(2008), 83--89.
	
	\bibitem{Gu} Z. J. Gu, A generalisation of diophantine tuples, Ramanujan J. 72 (2026), Article 79.
	
	\bibitem{HTZ} B. He, A. Togb\'e and V. Ziegler, There is no Diophantine quintuple, Trans. Am. Math. Soc. 371 (2019), 6665--6709.
	
	\bibitem{KYY} S. Kim, C. H. Yip and S. Yoo, Multiplicative structure of shifted multiplicative subgroups and its applications to Diophantine tuples, Canad. J. Math., to appear.
	
	\bibitem{LN} R. Lidl and H. Niederreiter, Finite Fields, 2nd edition, Cambridge University Press, Cambridge, 1997.
	
	\bibitem{MR} N. Mani and S. Rubinstein-Salzedo, Diophantine tuples over $\mathbb{Z}_p$, Acta Arith. 197 (2021), 331--351.
	
	\bibitem{Robin} G. Robin, Estimation de la fonction de Tchebychef $\theta$ sur le $k$-i\`eme nombre premier et grandes valeurs de
	la fonction $\omega(n)$ nombre de diviseurs premiers de $n$, Acta Arith. 42 (1983), 367--389.
	
	\bibitem{RS} J. B. Rosser and L. Schoenfeld, Approximate formulas for some functions of prime numbers, Illinois J. Math. 6 (1962), 64--94.
	
	\bibitem{Shparlinski} I. E. Shparlinski, On the number of Diophantine m-tuples in finite fields, Finite Fields Appl. 90 (2023), Article 102241.
	
	\bibitem{TY}  K. M. Tsang and C. H. Yip, Bipartite Diophantine tuples and their applications, Res. Number Theory Paper 12 (2026), Article 16. 
	\end{thebibliography}
\end{document}